\documentclass[oneside,11pt]{amsart}
\usepackage{amsmath}
\usepackage{amssymb}
\usepackage{graphicx}
\usepackage{tikz}
\usepackage{color}

\newtheorem{thm}{Theorem}[section]
\newtheorem{prop}[thm]{Proposition}
\newtheorem{lem}[thm]{Lemma}
\newtheorem{cor}[thm]{Corollary}

\newtheorem{conj}[thm]{Conjecture}

\theoremstyle{definition}
\newtheorem{defn}[thm]{Definition}
\newtheorem{rem}[thm]{Remark}

\newtheorem{exmp}[thm]{Example}
\newtheorem{ques}[thm]{Question}

\newtheorem{fact}[thm]{Fact}

\newcommand{\abs}[1]{\lvert{#1}\rvert}

\renewcommand{\bar}[1]{\overline{#1}}

\newcommand{\bigset}[2]{ \bigl\{ \, {#1} \bigm| {#2} \, \bigr\} }

\newcommand{\field}[1]{\mathbb{#1}}
\newcommand{\Z}{\field{Z}}

\newcommand{\Q}{\field{Q}}
\newcommand{\N}{\field{N}}

\newcommand{\HH}{\field{H}}

\DeclareMathOperator{\Isom}{Isom}

\DeclareMathOperator{\Aut}{Aut}

\DeclareMathOperator{\CAT}{CAT}

\usepackage{ifthen}

\newcommand{\showcomments}{yes}

\newsavebox{\commentbox}
{\ifthenelse{\equal{\showcomments}{yes}}%
{\footnotemark
    \begin{lrbox}{\commentbox}
    \begin{minipage}[t]{1.25in}\raggedright\sffamily\upshape\tiny
    \footnotemark[\arabic{footnote}]}
{\begin{lrbox}{\commentbox}}}%
{\ifthenelse{\equal{\showcomments}{yes}}%
{\end{minipage}\end{lrbox}\marginpar{\usebox{\commentbox}}}
{\end{lrbox}}}

\begin{document}

\title{Quasi-isometries of pairs: surfaces in graph manifolds}

\author{Hoang Thanh Nguyen}
\address{Dept.\ of Mathematical Sciences\\
University of Wisconsin--Milwaukee\\
P.O.~Box 413\\
Milwaukee, WI 53201\\
USA}
\email{nguyen36@uwm.edu}
\date{\today}

\begin{abstract}
We show there exists a closed graph manifold $N$ and infinitely many non-separable, horizontal surfaces $ \{S_{n} \looparrowright N\}_{n \in \N}$ such that there does not exist a quasi-isometry $\pi_1(N) \to \pi_1(N)$ taking $\pi_1(S_{n})$ to $\pi_1(S_{m})$ within a finite Hausdorff distance when $n \neq m$.
\end{abstract}

 % Geometric group theory
\maketitle
%%%%%%%%%%%%%%%%%%%%%%%%%%%%%%%%%%%%%%%%%%%%%%%%%

\section{Introduction}
\label{sec:Introduction}
A finitely generated group $G$ can be considered as a metric space when we equip $G$ with the word metric from a finite generating set. With different finite generating sets on $G$ we have different metrics on $G$, however such  metric spaces are unique up to \emph{quasi-isometric} equivalence. The notion of quasi-isometry that ignores small scale details is especially significant in geometric group theory following the work of Gromov. Two quasi-isometries of $G$ are called equivalent if they are within finite distance from each other. The group of quasi-isometries of $G$, denoted by $QI(G)$ is the set of equivalence classes of quasi-isometries $G \to G$ with the canonical operation (i.e, composition of maps). We note that if two finitely generated groups are quasi-isometric then their corresponding quasi-isometry groups are isomorphic.

A compact, orientable, irreducible $3$--manifold $N$ with empty or toroidal boundary is \emph{geometric} if its interior admits a geometric structure in the sense of Thurston. A non-geometric manifold $N$ is called a \emph{graph manifold} if all the blocks in its JSJ decomposition are Seifert fibered spaces. Quasi-isometric classification of graph manifolds has been studied by Kapovich-Leeb \cite{KL98} and a complete quasi-isometric classification for fundamental groups of graph manifolds is given by Behrstock-Neumann in \cite{NeuBeh08}. In particular, Behrstock-Neumann proved that the fundamental groups of all closed graph manifolds are quasi-isometric.  Thus, there is exactly one quasi-isometry group of fundamental group of closed graph manifolds. When $N$ is a closed graph manifold, the construction of quasi-isometries of Behrstock-Neumann is very flexible and produces many quasi-isometric classes in $QI(\pi_1(N))$.

A subgroup $H \le G$ is called \emph{separable} if for any $g \in G-H$ there exists a finite index subgroup $K \le G$ such that $H \le K$ and $g \notin K$. A horizontal surface $S \looparrowright N$ in a graph manifold $N$ is called \emph{separable} if $\pi_1(S)$ is a separable subgroup  in $\pi_1(N)$. Recently, Hruska and the author \cite{Hruska-Nguyen} show that the distortion of $\pi_1(S)$ in $\pi_1(N)$ is quadratically distorted whenever the surface is separable and is exponentially distorted otherwise. Therefore, there does not exist any quasi-isometry of fundamental groups of closed graph manifolds mapping a separable, horizontal surface to a non-separable, horizontal surface.

The purpose of this paper is trying to understand whether such quasi-isometries exist which map any separable (resp.\.non-separable) surface to another separable (resp.\,non-separable) surface.

If the two surfaces are both separable or both non-separable, then the subgroup distortion is not an useful quasi-isometric invariant to look at for the purpose above. Beside subgroup distortion, we remark that there are several other key quasi-isometric invariants of a pair $(G,H)$ in literature such as upper, lower relative divergence \cite{Hung15}, and $k$--volume distortion ($k \ge 1$) \cite{Bennett2011}, \cite{GerstenArea}. However, again the $k$--volume distortion is not an useful quasi-isometric invariant for our purpose above. We show that $k$--volume distortion of the surface subgroup in the $3$--manifold group is always trivial if $k \ge 3$, linear when $k=2$, and quadratic (resp.\,exponential) when $k=1$ and the surface is separable (resp.\,non-separable). Relative divergence which is introduced by Tran \cite{Hung15} is quite technical and difficult to compute in general. 
Recent work of Tran \cite{Tran17} allows us to show that the upper (lower) relative divergence of a separable, horizontal surface in a graph manifold is quadratic (linear) (see Appendix), however the author does not know the upper relative divergence and lower relative divergence in the non-separable case.

So far, none of the invariants discussed above can distinguish between two separable surfaces or two non-separable surfaces. The following questions are natural to ask.

\begin{ques}
\label{ques:non-separablequestion}
Given a closed graph manifold $N$. Are all pairs $\bigl ( \pi_1(N), \pi_1(S) \bigr)$ and $\bigl ( \pi_1(N), \pi_1(S') \bigr)$ with $S$ and $S'$ non-separable in $N$, quasi-isometric? 
\end{ques}

\begin{ques}
\label{ques:separablequestion}
Given a closed graph manifold $N$. Are all pairs $\bigl ( \pi_1(N), \pi_1(S) \bigr)$ and $\bigl ( \pi_1(N), \pi_1(S') \bigr)$ with $S$ and $S'$ separable in $N$, quasi-isometric? 
\end{ques}

In the questions above, two pairs $(G,H)$ and $(G',H')$ with $H\le G$ and $H'\le G'$ are \emph{quasi-isometric} if there is a quasi-isometry $G \to G'$ mapping  $H$ to $H'$ within a finite Hausdorff distance.

In this paper we show the answer to the Question~\ref{ques:non-separablequestion} is no.  We give examples of non-separable surfaces where such quasi-isometries never exist. The main theorem of this paper is the following

\begin{thm}
\label{thm:main}
There exists a closed simple graph manifold $N$ and infinitely many non-separable, horizontal surface $S_{n} \looparrowright N$ such that none of the pairs $\bigl (\pi_1(N),\pi_1(S_{n}) \bigr )$ and $\bigl (\pi_1(N),\pi_1(S_{m}) \bigr )$ are quasi-isometric when $n \neq m$
\end{thm}

No example is currently known for separable, horizontal surfaces $S \looparrowright N$ and $S' \looparrowright N$ such that $\bigl (\pi_1(N),\pi_1(S) \bigr )$ and $\bigl (\pi_1(N),\pi_1(S') \bigr )$ are not quasi-isometric.
Thus, we propose the following conjecture.
\begin{conj}
\label{conj}
Let $N$ and $M$ be two closed graph manifolds. Let $S \looparrowright N$ and $S' \looparrowright M$ be two separable, horizontal surfaces. Then there is a quasi-isometry from $\pi_1(N)$ to $\pi_1(M)$ mapping $\pi_1(S)$ to $\pi_1(S')$ in a finite Hausdorff distance.
\end{conj}

Although this paper deals only with graph manifolds, it is interesting to work on other classes of $3$--manifolds. We note that if $\Gamma \le \Isom(\HH^n)$ ($n \ge 3$) is a nonuniform lattice and $\Gamma$ is torsion-free then the quotient space $N = \HH^{n}/ \Gamma$ is a hyperbolic manifold of finite volume.
In the setting of non-uniform lattices in $\Isom(\HH^n)$ ($n \ge 3$), quasi-isometry of pairs can be intepreted via algebraic properties of groups (see Theorem~\ref{thm:uniformlatticequasi})
thank to Schwartz Rigidity Theorem.
\begin{thm}
\label{thm:uniformlatticequasi}
Let $\Gamma$ and $\Gamma'$ be two non-uniform lattices of $\Isom(\HH^{n})$ with $n \ge 3$. Let $H$ and $H'$ be two finitely generated subgroups of $\Gamma$ and $\Gamma'$ respectively. Then $(\Gamma,H)$ and $(\Gamma',H')$ are quasi-isometric  if and only if there exists $g \in \Isom(\HH^{n})$ such that $gHg^{-1}$ and $H'$ are commensurable  as well as $g\Gamma g^{-1}$ and  $\Gamma'$ are commensurable.
\end{thm}

%%%%%%%%%%%%%%%%%%%%%%%%%%%%%%%%%%%%%%%%%%%%%%%%

\subsection*{Acknowledgments}
I would like to thank my adviser Chris Hruska for all his
help and advice throughout this paper, and many thank to Walter Neumann,  Hung Cong Tran and Prayagdeep Parija for useful discussion.

%%%%%%%%%%%%%%%%%%%%%%%%%%%%%%%%%%%%%%%%%%%%%%%%
\section{Quasi-isometry of the pairs of spaces}
In this section, we review some notions in geometric group theory.

Let $(X,d)$ be a metric space, and $\gamma$ a path in $X$. We denote the length of $\gamma$ by $\abs{\gamma}$.
If $A$ and $B$ are subsets of $X$, the \emph{Hausdorff distance} between $A$ and $B$ is
\[
d_{\mathcal{H}}(A,B) = \inf \bigset{r}{A \subseteq \mathcal{N}_{r}(B) \, \text{and $B \subseteq \mathcal{N}_{r}(A)$}}
\] where $\mathcal{N}_{r}(C)$ denotes the $r$--neighborhood of a subset $C$.

\begin{defn}[Quasi--isometry] 
Let $(X_{1},d_{1})$ and $(X_{2},d_{2})$ be metric spaces. A (not necessarily continuous) map $f \colon X_{1} \to X_{2}$ is an \emph{$(L,C)$--quasi-isometric embedding} if there exist constants $L\geq 1$ and $C\geq 0$ such that for all $x,y \in X_{1}$ we have
\[
   \frac{1}{L} \, d_{1}(x,y) - C
   \leq d_{2} \bigl( f(x),f(y) \bigr)
   \leq L\,d_1(x,y) + C.
\]
If, in addition, there exits a constant $D\geq{0}$ such that every point of $X_{2}$ lies in the $D$--neighborhood of the image of $f$, then $f$ is an \emph{$(L,C)$--quasi-isometry}. When such a map exists, $X_{1}$ and $X_{2}$ are \emph{quasi-isometric}.
\end{defn}

\begin{defn}
\label{defn:quasipairs}
Let $X$ and $Y$ be metric spaces. Let $A$ be a subspace of $X$ and $B$ a subspace of $Y$. Two pairs of spaces $(X,A)$ and $(Y,B)$ is called \emph{quasi-isometric} if there exists an $(L,C)$--quasi-isometry $f \colon X \to Y$ such that $f(A) \subseteq B$ and $B \subseteq \mathcal{N}_{C}(f(A))$. We call the map $f$ an $(L,C)$--\emph{quasi-isometry of pairs}. We denote $(X,A) \sim (Y,B)$ if $(X,A)$ and $(Y,B)$ are quasi-isometric, and $(X,A) \nsim (Y,B)$ otherwise.
\end{defn}

\begin{defn}
Let $G$ and $G'$ be finitely generated groups with finite generating sets $\mathcal{S}$ and $\mathcal{S'}$ respectively. Let $\Gamma(G,\mathcal{S})$ and $\Gamma(G',\mathcal{S'})$ be the Cayley graphs of $(G,\mathcal{S})$ and $(G',\mathcal{S'})$ respectively.  Let $H$ be a subgroups of $G$ and $H'$ a subgroup of $G'$. We say $(G,H)$ and $(G',H')$ are quasi-isometric if the pairs of spaces $\bigr (\Gamma(G,\mathcal{S}), H \bigr )$ and $\bigr (\Gamma(G',\mathcal{S'}), H' \bigr )$ are quasi-isometric.
\end{defn}

\begin{rem}
\label{rem:simple}
\begin{enumerate}
\item We note that $\sim$ is an equivalent relation, and
the relation $(G,H) \sim (G',H')$ is independent of choices of generating sets for the groups.
\item  If there is a quasi-isometry $f \colon X \to Y$ such that $d_{\mathcal{H}}\bigl (f(A),B \bigr )$ is finite  then $(X,A)$ and $(Y,B)$ are quasi-isometric.
\end{enumerate}
\end{rem}

\begin{exmp}
\label{exmp:mainexample}
\begin{enumerate}
    \item Let $H \le G_1$ be finitely generated subgroups of a finitely generated group $G$. Suppose that $G_1$ is a finite index subgroup in $G$. Then $(G,H) \sim (G_1,H)$.
    \item
    \label{item:finite} $(G,H) \sim (G,G)$ if and only if $H$ is a finite index subgroup of $G$.
    \item Let $F_{n}$ be the free group on $n$ generators with $n\ge2$. It is well-known that there are two isomorphic subgroups $H$ and $K$ of $F_{n}$ such that $H$ is a finite index subgroup of $F_{n}$ and $K$ is an infinite index subgroup of $F_{n}$. It follows from (\ref{item:finite}) that $(F_{n},H) \nsim (F_{n},K)$.
   \item \label{item:Seifertquasi} Let $M_1$ and $M_2$ be Seifert fibered spaces with the base surfaces have negative Euler characteristic and nonempty boundary. Let $S_1$ and $S_2$ be two properly immersed $\pi_1$--injective horizontal surfaces in $M_1$ and $M_2$ respectively. The immersion $S_{j} \looparrowright M_{j}$ lifts to an embedding in a finite cover $S_{j} \times S^1$ of $M_{j}$ (see Lemma~2.1 \cite{RW98}). Since $\pi_1(S_j)$ is free, there exists a quasi-isometry $\pi_1(S_1) \to \pi_1(S_2)$. It follows easily that  $\bigl (\pi_1(S_{1} \times S^1 ) , \pi_1(S_1) \bigr ) \sim \bigl (\pi_1(S_{2} \times S^1 ) , \pi_1(S_2) \bigr )$. As a consequence
    $\bigl ( \pi_1(M_1), \pi_1(S_1) \bigr ) \sim \bigl (\pi_1(M_2), \pi_1(S_2) \bigr )$.
\end{enumerate}
\end{exmp}

\begin{defn}[Commensurable]
Let $G$ be a group. Two subgroups $H$ and $K$ of $G$ are called \emph{commensurable} if $H \cap K$ is a finite index subgroup of both $H$ and $K$.
\end{defn}

We use the following lemma in the proof of Theorem~\ref{thm:uniformlatticequasi}.
\begin{lem}[Corollary~2.4 \cite{MSW11}]
\label{lem:comm}
Two subgroups $H$ and $K$ are commensurable in a finitely generated group $G$ if and only if $H$ is within a finite Hausdorff distance with $K$.
\end{lem}

\begin{proof}[Proof of Theorem~\ref{thm:uniformlatticequasi}]
We are going to prove sufficiency. Let equip $\Gamma$ and $\Gamma'$ with word metrics.
Let $f \colon \Gamma \to \Gamma'$ be a quasi-isometry such that $f(H)$ is within a finite Hausdorff distance with $H'$ with respect to the word metric on $\Gamma'$. By Schwartz Rigidity Theorem \cite{Schwartz95} (see also, for example, Theorem~24.1 \cite{DK18}), there exists $g \in \Isom(\HH^{n})$ such that the following holds:
\begin{enumerate}
    \item $g \Gamma g^{-1}$ and $\Gamma'$ are commensurable.
    \item Let $G$ be the subgroup of $\Isom(\HH^n)$ generated by two subgroups $g\Gamma g^{-1}$ and $\Gamma'$. We note that $G$ is a finitely generated subgroup. We equip $G$ with a word metric. For each $\gamma \in \Gamma$, choose $y_{\gamma} \in \Gamma'$ which is nearest to $g \gamma g^{-1}$ with respect to the word metric on $G$. Then the map $q_{g} \colon \Gamma \to \Gamma'$ which sends $\gamma$ to $y_{\gamma}$ is a quasi-isometry and $q_{g}$ is within finite distance from $f$.
\end{enumerate}
We will need to show $gHg^{-1}$ and $H'$ are commensurable. Let $d_{\mathcal{H}}$ denote for the Hausdorff distance of any two subsets of $G$ with respect to the given word metric on $G$. By the definition of $q_{g}$ we have $d_{\mathcal{H}}(gHg^{-1}, q_{g}(H)) < \infty$ and $d_{\mathcal{H}}(q_{g}(H),f(H)) < \infty$. By the assumption of $f$ we have $d_{\mathcal{H}}(f(H),H') < \infty$. We use the triangle inequality to get that
\[
d_{\mathcal{H}}(gHg^{-1},H') \le d_{\mathcal{H}}(gHg^{-1},q_{g}(H)) + d_{\mathcal{H}}(q_{g}(H),f(H)) + d_{\mathcal{H}}(f(H),H') < \infty
.\] Thus, $gHg^{-1}$ and $H'$ are commensurable by Lemma~\ref{lem:comm}. 

We now are going to prove necessity. Suppose that there exists $g \in \Isom(\HH^n)$ such that $gHg^{-1}$ and $H'$ are commensurable as well as $g\Gamma g^{-1}$ and $\Gamma'$ are commensurable. Let $G$ be the subgroup of $\Isom(\HH^n)$ generated by two subgroups $g\Gamma g^{-1}$ and $\Gamma'$. We equip $G$ with a word metric $d$, and with respect to this metric we denote $d_{\mathcal{H}}$ the Hausdorff distance of any two subsets of $G$. Since  $gHg^{-1}$ and $H'$ are commensurable as well as $g\Gamma g^{-1}$ and $\Gamma'$ are commensurable, it follows that there is $R >0$ such that $d_{\mathcal{H}}(gHg^{-1},H') \le R$ and $d_{\mathcal{H}}(\Gamma',g \Gamma g^{-1}) \le R$. For each $\gamma \in \Gamma$, choose  an element $q_{g}(\gamma)$ in $\Gamma'$ such that $d(q_{g}(\gamma),g\gamma g^{-1}) \le R$. We thus define the map $q_{g} \colon \Gamma \to \Gamma'$. Since the map $\Gamma \to g\Gamma g^{-1}$ which sends $\gamma$ to $g\gamma g^{-1}$ is a quasi-isometry and $q_{g}$ is within finite distance from the map $\Gamma \to g\Gamma g^{-1}$,  it follows that $q_{g}$ is a quasi-isometry. From the definition of $q_{g}$, it is obvious that $d_{\mathcal{H}}(q_{g}(H),gHg^{-1}) \le R$. We use the facts $d_{\mathcal{H}}(q_{g}(H),gHg^{-1}) \le R$ and $d_{\mathcal{H}}(gHg^{-1},H') \le R$ to get that
\[
d_{\mathcal{H}}(q_{g}(H),H') \le d_{\mathcal{H}}(q_{g}(H),gHg^{-1}) + d_{\mathcal{H}}(gHg^{-1},H') \le 2R
\]
Thus $(\Gamma,H)$ and $(\Gamma,H')$ are quasi-isometric via the map $q_{g}$.
\end{proof}

We use the following lemma in the proof of Theorem~\ref{thm:main}.
\begin{lem}
\label{lem:invariant of length}
Let $G$ and $G'$ be finitely generated groups with finite generating sets $\mathcal{S}$ and $\mathcal{S}'$ respectively. Let $H$ and $H'$ be finitely generated subgroups of $G$ and $G'$ with finite generating sets $\mathcal{A}$ and $\mathcal{A}'$ respectively such that $\mathcal{A} \subseteq \mathcal{S}$ and $\mathcal{A}' \subseteq \mathcal{S}'$. Let $\varphi \colon \bigl (\Gamma(G,\mathcal{S}),H \bigr ) \to \bigl (\Gamma(G',\mathcal{S}'),H' \bigr )$ be an $(L,C)$--quasi-isometry of pairs. Then there exists a constant $L'$ such that
\[
\frac{\abs{h}_{\mathcal{A}}}{L'} - L' \le \abs{\varphi(h)}_{\mathcal{A'}} \le L'\abs{h}_{\mathcal{A}} + L'
\] for all $h \in H$.
\end{lem}
%%%%%%%%%%%%%%%%%%%%%%%%%%%%%%%%%%%%%%%%%%%%%%% PROOF
\begin{proof}
Let $e$ and $e'$ be the identity elements in the groups $G$ and $G'$. For any $h \in H$, let $\alpha$ be a geodesic in the Cayley graph $\Gamma(H', \mathcal{A'})$ connecting $\varphi(e)$ to $\varphi(h)$. We denote $\varphi(e) = y_0,y_1,\dots,y_{m} = \varphi(h)$ be the sequence of vertices belong to $\alpha$. Since $H' \subset N_{C}(\varphi(H))$, there exists $x_i \in H$ such that $d_{\mathcal{S'}}(\varphi(x_i),y_i) \le C$ with $x_0 =e $ and $x_m = h$. Moreover, we have 
\[
d_{\mathcal{S}}(x_i,x_{i+1}) \le Ld_{\mathcal{S'}}(\varphi(x_i),\varphi(x_{i+1})) +C \le L+C.
\]
Since $G$ is locally finite, it follows that there exists a constant $R$ depending on $L$ and $C$ such that $d_{\mathcal{A}}(x_i,x_{i+1}) \le R$. It is obvious that $\abs{h}_{\mathcal{A}} = d_{\mathcal{A}}(e,h) \le mR$, thus $\abs{h}_{\mathcal{A}}= d_{\mathcal{A}}(e,h) \le R d_{\mathcal{A'}}(\varphi(e),\varphi(h)) \le R\abs{\varphi(e)}_{\mathcal{A'}} + R\abs{\varphi(h)}_{\mathcal{A'}}$. Let $\bar{\varphi}$ be a quasi-inverse of $\varphi$, by a similar argument it is not hard to see that constants $L'$ and $C'$ exist.
\end{proof}
%%%%%%%%%%%%%%%%%%%%%%%%%%%%%%%%%%%%%%%%%%%%%%%%%

It is well known that a group acting properly, cocompactly, and isometrically on a geodesic space is quasi-isometric to the space. The following corollary of this fact allows us to show two pairs of (group, subgroup) are quasi-iometric using the geometric properties of spaces in place of words metrics.

\begin{cor}
\label{cor: geometric relative quasi to group relative quasi}
Let $X_i$ and $Y_i$ be compact geodesic spaces, and let $g\colon{(Y_i,y_i)} \to (X_i,x_i)$ be $\pi_1$--injective with $i = 1,2$. We lift the metrics on $X_i$ and $Y_i$ to geodesic metrics on the universal covers $\tilde{X}_i$ and $\tilde{Y}_i$ respectively. Let $G_i = \pi_1(X_i,x_i)$ and $H_i = g_{*} \bigl( \pi_1(Y_i,y_i) \bigr)$. Then $(\tilde{X}_1,\tilde{Y}_1)$ and $(\tilde{X}_2,\tilde{Y}_2)$ are quasi-isometric if and only if $(G_1,H_1)$ and $(G_2,H_2)$ are quasi-isometric.
\end{cor}

\section{Graph manifolds and horizontal surfaces}

In this section, we briefly review backgrounds about graph manifolds and horizontal surfaces and give some lemmas which will be used in the proof of Theorem~\ref{thm:main} in Section~\ref{sec:proofthm}. All $3$--manifolds are always assumed to be compact, orientable, irreducible with empty or toroidal boundary. A non-geometric manifold $N$ is called \emph{graph manifold} if all the blocks in its JSJ decomposition are Seifert fibered spaces.

\begin{defn}
A \emph{simple graph manifold} $N$ is a graph manifold with the following properties: Each Seifert block is a trivial circle bundle over an orientable surface of negative Euler characteristic. The intersection numbers of fibers of adjacent Seifert blocks have absolute value 1. It was shown by Kapovich and Leeb that any graph manifold $N$ has a finite cover $\hat{N}$ that is a simple graph manifold \cite{KL98}.
\end{defn}

\begin{defn}
Let $M$ be a Seifert manifold with boundary.
A \emph{horizontal surface} in $M$ is a properly immersed surface $g\colon B\looparrowright M$ where $B$ is a compact surface with boundary and the image $g(B)$ is transverse to the Seifert fibration.
A \emph{horizontal surface} in a graph manifold $N$ is a properly immersed surface $g\colon S\looparrowright{N}$ such that for each Seifert block $M$, the intersection $g(S)\cap{M}$ is a horizontal surface in $M$.
\end{defn}

\begin{defn}[Slope and Spirality]
\label{defn:DilationSlopes}
Let $g \colon (S,s_{0}) \looparrowright (N,x_{0})$ be a horizontal surface in a simple graph manifold $N$. Let $\mathcal{T}$ be the union of tori in the JSJ decomposition  of $N$. We denote the collecton of circles in $g^{-1}(\mathcal{T})$ by $\mathcal{C}_{g}$.
With respect to $\mathcal{T}$, let $\Omega$ be the dual graph of $N$. With respect to $g^{-1}(\mathcal{T})$, let $\Omega_{S}$ be the dual graph of $S$. Let $O(\Omega)$ and $O(\Omega_{S})$  be the set of oriented edges in the graphs $\Omega$ and $\Omega_{S}$ respectively. For each oriented edge $e \in O(\Omega_S)$, the \emph{slope} of $e$, denoted by $\mathbf{sl}(e)$, is defined as follows. The initial vertex (resp, ternimal vertex) of $e$ is corresponding to a component in $S - \mathcal{T}_g$ and this component is mapped into a Seifert block, denoted by $\overleftarrow{M}$ (resp, $\overrightarrow{M}$).
 Let $c$ be the circle in $\mathcal{T}_{g}$ corresponding to the oriented edge $e$.
The image of the circle $g(c)$ in $N$ lies in a JSJ torus $T$ obtained by gluing a boundary torus $\overleftarrow{T}$ of $\overleftarrow{M}$ to a boundary torus $\overrightarrow{T}$ of $\overrightarrow{M}$.
Let $\overleftarrow{\beta}$ and $\overrightarrow{\beta}$ be fibers of $\overleftarrow{M}$ and $\overrightarrow{M}$ in the torus $T$.
The $1$--cycles $[\overleftarrow{\beta}]$ and $[\overrightarrow{\beta}]$ generate the integral homology group $H_1(T) \cong \Z^2$, so there exist integers $a$ and $b$ such that
\[
   \bigl[g(c)\bigr]=a[\overleftarrow{\beta}]+b[\overrightarrow{\beta}]
   \quad \text{in} \quad
   H_1(T).
\]
Since the surface $S$ is horizontal, neither of $a$ and $b$ are equal to $0$. We set \[\mathbf{sl}(e) = |b/a|\]
Since we take the absolute value of $b/a$, we note that the way we define $\mathbf{sl}(e)$ does not really depend on choices of orientations of $c$, $\overleftarrow{\beta}$ and $\overrightarrow{\beta}$. Also, for the opposite oriented edge $-e$ we have $\mathbf{sl}(-e) = |a/b|$. In other words, $\mathbf{sl}(e) \cdot \mathbf{sl}(-e) = 1$.

The \emph{spirality} of $S$ in $N$ is a homomorphism $w \colon \pi_1(S,s_0) \to \Q_{+}^{*}$ defined as follows. Choose $[\gamma] \in \pi_1(S,s_0)$ such that $\gamma$ is transverse to $\mathcal{C}_g$.  
In the trivial case that $\gamma$ is disjoint from the curves of the collection $\mathcal{C}_g$, we set $w(\gamma) = 1$.  Let us assume now that this intersection is nonempty.
Then $\mathcal{C}_g$ subdivides $\gamma$ into a concatenation $\gamma_1\cdots\gamma_m$ with the following properties.
Each path $\gamma_i$ starts on a circle $c_i \in \mathcal{C}_g$ and ends on the circle $c_{i+1}$. The path $\gamma$ determines an oriented cycle $e_1 \cdot e_{2} \cdots e_{m}$ in the graph $\Omega_S$.
We let $w(\gamma)$ be the rational number
$\prod_{i=1}^{m} \mathbf{sl}(e_i))$. We remark that the sequence $\{\mathbf{sl}(e_i)\}$ depends on the choice of curve $\gamma$, but the product $w(\gamma)$ depends only on the homology class of $\gamma$ since $\mathbf{sl}(e) \cdot \mathbf{sl}(-e) = 1$ for each $e \in O(\Omega_S)$. 
\end{defn}

\begin{rem}
\label{rem:main}
\begin{enumerate}
    \item
    \label{item:1}
    The notion of spirality is due to Rubinstein-Wang \cite{RW98}.
    The name ``spirality'' was introduced by Yi Liu \cite{YiLiu2017}. The concept has also been called ``dilation'' by Woodhouse \cite{Woodhouse16} and Hruska-Nguyen \cite{Hruska-Nguyen}. We note that the horizontal $g \colon S \looparrowright N$ is separable (equivalent to virtually embedded) if and only if the spirality of $S$ in $N$ is a trivial homomorphism (see Theorem~2.3 in \cite{RW98}).
\item The \emph{governor} of the horizontal surface $g \colon S \looparrowright N$ is the quantity 
$\epsilon = \epsilon(g) = \max \bigset{\mathbf{sl}(e)}{\textup{$e$ is an oriented edge in $O(\Omega_S)$}}$. From (\ref{item:1}), we note that if $g \colon S \looparrowright N$ is non-separable then $\epsilon >1$.
\end{enumerate}
\end{rem}

In the following, we use the notation $[\alpha] \wedge [\beta]$ to denote the algebraic intersection number of two oriented closed curves $\alpha$ and $\beta$ in a torus $T$ with respect to some choosen orientation on $T$.

Let $g \colon S \looparrowright N$ be a horizontal surface in a simple graph manifold $N$ given by Definition~\ref{defn:DilationSlopes}. We have the following lemma that give us an alternative way to compute slopes.

\begin{lem}
\label{lem:alternativeslope}
For each oriented edge $e$ in $O(\Omega_S)$. Let $\overleftarrow{B}$ and $\overrightarrow{B}$ be the components in $S - g^{-1}(\mathcal{T})$ corresponding to the initial and terminal vertices of $e$. These components $\overleftarrow{B}$ and $\overrightarrow{B}$ are mapped by $g$ into Seifert blocks $\overleftarrow{M}$ and $\overrightarrow{M}$ respectively. Let $c$ be the circle in $g^{-1}(\mathcal{T})$ corresponding to the edge $e$ obtained by gluing a boundary circle $\overleftarrow{c}$ of $\overleftarrow{B}$ to a boundary circle $\overrightarrow{c}$ of $\overrightarrow{B}$. The image $g(c)$ in $N$ lies in a JSJ torus $T$ obtained by gluing a boundary torus $\overleftarrow{T}$ of $\overleftarrow{M}$ to a boundary torus $\overrightarrow{T}$ of $\overrightarrow{M}$.  Let $\overleftarrow{\beta}$ and $\overrightarrow{\beta}$ be fibers of $\overleftarrow{M}$ and $\overrightarrow{M}$ in the torus $T$. Let $\overleftarrow{\alpha}$ and $\overrightarrow{\alpha}$ be oriented simple closed curves in $\overleftarrow{T}$ and $\overrightarrow{T}$ respectively such that $\Bigl |[\overleftarrow{\alpha}] \wedge [\overleftarrow{\beta}] \Bigr | =1$ and $\Bigl | [\overrightarrow{\alpha}] \wedge [\overrightarrow{\beta}] \Bigr | =1$. If $[g(c)] = m[\overleftarrow{\alpha}] + n[\overleftarrow{\beta}]$ and $[g(c)] = m'[\overrightarrow{\alpha}] + n'[\overrightarrow{\beta}]$ for some intergers $m,n,m'$ and $n'$. Then $\mathbf{sl}(e) = \abs{m/m'}$.
\end{lem}
\begin{proof}
Let $a$ and $b$ be integers such that $[g(c)] = a[\overleftarrow{\beta}] + b[\overrightarrow{\beta}]$, By the definition of slope, we have $\mathbf{sl}(e) = \abs{b/a}$.  We  use the distributive law and the scalar multiplication law of algebraic intersection together with the facts $[\overleftarrow{\beta}] \wedge [\overleftarrow{\beta}] =0$, $\Bigl |[\overleftarrow{\alpha}] \wedge [\overleftarrow{\beta}] \Bigr | =1$ and $\Bigl |[\overleftarrow{\beta}] \wedge [\overrightarrow{\beta}] \Bigr | =1$ to get that
\[
\Bigl |[g(c)] \wedge [\overleftarrow{\beta}] \Bigr | = \Bigl |m[\overleftarrow{\alpha}]\wedge[\overleftarrow{\beta}] + n[\overleftarrow{\beta}]\wedge[\overleftarrow{\beta}] \Bigr | = \abs{m}\] and \[\Bigl |[g(c)] \wedge [\overleftarrow{\beta}] \Bigr | = \Bigl | a[\overleftarrow{\beta}] \wedge [\overleftarrow{\beta}] + b[\overrightarrow{\beta}] \wedge [\overleftarrow{\beta}] \Bigr | = \abs{b}\] Thus, $\abs{m} = \abs{b}$.

Similarly, we get that
\[\Bigl |[g(c)] \wedge [\overrightarrow{\beta}] \Bigr | = \Bigl |m'[\overrightarrow{\alpha}]\wedge[\overrightarrow{\beta}] + n'[\overrightarrow{\beta}]\wedge[\overrightarrow{\beta}] \Bigr | = \abs{m'}\] and \[\Bigl |[g(c)] \wedge [\overrightarrow{\beta}] \Bigr | = \Bigl | a[\overleftarrow{\beta}] \wedge [\overrightarrow{\beta}] + b[\overrightarrow{\beta}] \wedge [\overrightarrow{\beta}] \Bigr | = \abs{a}\] Thus, $\abs{m'} = \abs{a}$. It follows that $\abs{b/a} = \abs{m/m.}$. Therefore $\mathbf{sl}(e) = \abs{m/m'}$.
\end{proof}

We use the following lemma in the construction of surfaces in Lemma~\ref{lem:existence}.
\begin{lem}[Lemma~2.2 in \cite{RW98}]
\label{lem:RW lemma}
Let $F$ be a surface with non-empty boundary, positive genus and $\chi(F) <0$. Let $M$ be the trivial Seifert fibered space $F \times S^1$. We fix orientations of the surface $F$ and the fiber $S^1$ of $M$. Let $\bigset{\alpha_i}{i =1,2,\dots ,t}$ be the collection of oriented boundary curves of the surface $F$. Let $T_i = T(\alpha_i,\beta_i)$ be the boundary torus of $M = F \times S^1$ where $\beta_i$ is a oriented fiber $S^1$ corresponding to the second factor of $M$ with $i = 1,2,\dots,n$. Suppose that $\bigset{c_{ij}}{j = 1,2}$ is a family of oriented simple closed curves on $T_i$ and 
\[
   \bigl[c_{ij}\bigr] = u_{ij}[\alpha_i] + v_{ij}[\beta_i]
   \quad \text{in} \quad
   H_1(T_i)
\] for some integers $u_{ij}$ and $v_{ij}$ with $u_{ij} >0$.

Then the union of family $\bigset{c_{ij}}{j = 1,2}$ is a boundary of a connected immersed orientable horizontal surface $S$ in $M$ if the following holds
\begin{enumerate}
\item $\sum_{i=1}^{n}\,\sum_{j=1}^{2}v_{ij} = 0$
\item There exists $u>0$ such that for all $i$ we have $u_{i1} + u_{i2} = u $.
\item $u\chi(F)$ is even.
\end{enumerate}
\end{lem}

\begin{rem}
\label{rem:genus}
Let $S \looparrowright M$ be the horizontal surface given by Lemma~\ref{lem:RW lemma}. By the construction, we note that the number of boundary components of $S$ is $2t$. We can compute the genus $x$ of $S$ as the follows. The composition of $S \looparrowright M$ with the projection of $M$ to $F$ yields a finite covering map $S \to F$ with degree $u$. Hence, $\chi(S) = u\chi(F)$. It follows that $2-2x -2t = u\chi(F)$. Thus, $x = \bigl (2 -2t -u\chi(F) \bigr )\bigl /2$.
\end{rem}

Let $\overleftarrow{\alpha}, \overleftarrow{\beta}, \overrightarrow{\alpha}$ and $\overrightarrow{\beta}$ be the copies of the circle $S^1$. Let $\overleftarrow{T} = \overleftarrow{\alpha} \times \overleftarrow{\beta}$ and $\overrightarrow{T} = \overrightarrow{\alpha} \times \overrightarrow{\beta}$. Let $J = \Bigl(\begin{matrix}
p&q \\ r&s
\end{matrix} \Bigr)$ be the $2$ by $2$ matrix such that $p,q,r,s \in \Z$, $q \neq 0$ and $ps - qr = -1$.
With respect to the matrix $J$, the basis $\{[\overleftarrow{\alpha}], [\overleftarrow{\beta}] \}$ in $H_{1}(\overleftarrow{T})$, and the basis $\{[\overrightarrow{\alpha}], [\overrightarrow{\beta}]\}$ in $H_{1}(\overrightarrow{T})$, there is a homeomorphism $h \colon \overleftarrow{T} \to \overrightarrow{T}$ such that $h_{*} \colon H_{1}(\overleftarrow{T}) \to H_{1}(\overrightarrow{T})$ has the matrix $J$ in the sense that 
\[
h_{*} \Bigl( a[\overleftarrow{\alpha}] + b[\overleftarrow{\beta}] \Bigr) = \Bigl ([\overrightarrow{\alpha}], [\overrightarrow{\beta}] \Bigr) \Bigl(\begin{matrix}
p&q \\ r&s
\end{matrix} \Bigr) \Bigl(\begin{matrix}
a \\ b
\end{matrix} \Bigr)
\]
In the rest of this paper, when we say we glue the torus $\overleftarrow{T}$ to the torus $\overrightarrow{T}$ via matrix $J$, we mean that the gluing map is the homeomorphism $h$.

\section{Constructing horizontal surfaces}
\label{sec:construngtingsurfaces}
In this section, we will construct a closed simple graph manifold $N$ and a collection of horizontal surfaces $\{S_{n} \looparrowright N\}$ such that when we pass to a specific subsequence then this subsequence satisfies the conclusion of Theorem~\ref{thm:main}. We also recall some facts from \cite{Hruska-Nguyen} that will be used in Section~\ref{sec:proofthm}.

\begin{lem}
\label{lem:existence}
There exists a closed simple graph manifold $N$ such that for any $n \in \N$, there exists a non-separable, horizontal surface $g_{n} \colon S_{n} \looparrowright N$ with the following properties.
\begin{enumerate}
    \item The governor of $g_{n} \colon S_{n} \looparrowright N$ is $\epsilon_{n} = 2n+1$
    \item Let $\mathcal{T}$ be the union of the JSJ tori of $N$. There exists a simple closed curve $\gamma_{n}$ in $S_{n}$ such that the geometric intersection number of $\gamma_{n}$ and  $g_{n}^{-1}(\mathcal{T})$ is $2$ and $w(\gamma_{n}) = (2n+1)^{2}$. 
\end{enumerate}
\end{lem}

\begin{proof}
We first construct a simple graph manifold $N'$ with non-empty boundary, and a non-serable horiontal surface $B_{n} \looparrowright N'$ (for each $n$) satisfying the conclusion of the lemma (we refer the reader to Figure~1 for an illustration.), and then obtain a closed simple graph manifold $N$ and a closed surface $S_n$ by doubling $N'$ and $B_n$ along their boundaries respectively. The construction here is inspired from Example~2.6 in \cite{RW98}.

Let $\overleftarrow{F}$ be the once punctured torus with the boundary circle denoted by $\overleftarrow{\alpha}$. Let $\overleftarrow{\beta}$ be the fiber factor of the trivial Seifert fibered space $\overleftarrow{F} \times S^1$. We denote the boundary torus of $\overleftarrow{F} \times S^1$ by $\overleftarrow{T_1}$. We fix orientations of $\overleftarrow{F}$ and $\overleftarrow{\beta}$.
 Let $\overrightarrow{F}$ be a twice punctured torus with two boundary circles denoted by $\overrightarrow{\alpha}$ and $\overrightarrow{\alpha'}$.  Let $\overrightarrow{\beta}$ be the fiber factor of the trivial Seifert fibered space $\overrightarrow{F} \times S^1$. The space $\overrightarrow{F} \times S^1$ has two boundary tori $\overrightarrow{T_1} = T(\overrightarrow{\alpha},\overrightarrow{\beta})$ and $\overrightarrow{T_2} = T(\overrightarrow{\alpha'},\overrightarrow{\beta})$. We fix orientations of $\overrightarrow{F}$ and $\overrightarrow{\beta}$.
 
Let $J$ be the matrix $ \Bigl(\begin{matrix}
1&1 \\ 2&1
\end{matrix} \Bigr)$.
Let $N'$ be the simple graph manifold obtained from gluing the boundary torus $\overleftarrow{T_1}$ of $\overleftarrow{F} \times S^1$ to the boundary torus $\overrightarrow{T_1}$ of $\overrightarrow{F} \times S^1$ via the gluing matrix $J$. We note that $N'$ is a simple graph manifold because of $\Bigl |[\overleftarrow{\beta}] \wedge [\overrightarrow{\beta}] \Bigr| =1$. To see this, we note that $[\overleftarrow{\beta}] = 0[\overleftarrow{\alpha}] + [\overleftarrow{\beta}]$ and hence in $H_{1}(\overrightarrow{T}_1)$ we have $[\overleftarrow{\beta}] = \Bigl([\overrightarrow{\alpha}], [\overrightarrow{\beta}] \Bigr)\Bigl(\begin{matrix}
1&1 \\ 2&1
\end{matrix} \Bigr) \Bigl(\begin{matrix}
0 \\ 1
\end{matrix} \Bigr) = [\overrightarrow{\alpha}] + [\overrightarrow{\beta}]$. Thus $\Bigl |[\overleftarrow{\beta}] \wedge [\overrightarrow{\beta}] \Bigr| = \Bigl |[\overrightarrow{\alpha}] \wedge [\overrightarrow{\beta}] \Bigr | =1$.

%For any $\delta >0$, let $n \in \N$ be any number such that $(2n+1)^{2} \ge \delta$.

Let $\overleftarrow{B}$ be the orientable surface with two boundaries $\overleftarrow{c_1}$ and $\overleftarrow{c_2}$ with $n + 1$ genus. Let $\overrightarrow{B}$ be the orientable surface with four boundaries $\overrightarrow{c_1}, \overrightarrow{c_2}, \overrightarrow{c_3}, \overrightarrow{c_4}$ and $2n +1$ genus.
Let $c_{1,1}$ and $c_{1,2}$ be oriented simple closed curves in $\overleftarrow{T_1}$ such that in  $ H_1(\overleftarrow{T_1})$ we have $[c_{1,1}] = [\overleftarrow{\alpha}] + 2n[\overleftarrow{\beta}]$ and $[c_{1,2}] = (2n+1)[\overleftarrow{\alpha}] - 2n[\overleftarrow{\beta}]$. 
Applying Lemma~\ref{lem:RW lemma} to $\overleftarrow{F}$, $\overleftarrow{T_1}$, $c_{1,1}$ and $c_{1,2}$, there is a horizontal surface $\overleftarrow{g} \colon \overleftarrow{B} \looparrowright \overleftarrow{F} \times S^1$ such that in $ H_1(\overleftarrow{T}_1)$ we have
\begin{equation}
\tag{$*$}
\label{eq:blacklozenge}
\begin{split}
\bigl[\overleftarrow{g}(\overleftarrow{c_{1}})\bigr] &= [c_{1,1}] = [\overleftarrow{\alpha}]+2n[\overleftarrow{\beta}]\\
   \bigl[\overleftarrow{g}(\overleftarrow{c_{2}})\bigr] &= [c_{1,2}] = (2n+1)[\overleftarrow{\alpha}] - 2n[\overleftarrow{\beta}]
 \end{split}
\end{equation}
Since the homeomorphism $\overleftarrow{T_1} \to \overrightarrow{T_1}$ is the gluing matrix $J$, it follows that in $H_{1}(\overrightarrow{T})$ we have $[c_{1,1}] = (2n+1)[\overrightarrow{\alpha}] + (2n+2)[\overrightarrow{\beta}]$ and $[c_{1,2}] = [\overrightarrow{\alpha}] + (2n+2)[\overrightarrow{\beta}]$.
Let $c_{2,1}$ and $c_{2,2}$ be oriented simple closed curves in $\overrightarrow{T_2}$ such that in $H_{1}(\overrightarrow{T_2})$ we have $[c_{2,1}] = [\alpha'] -(2n+2)[\overrightarrow{\beta}]$ and $[c_{2,2}] = (2n+1)[\overrightarrow{\alpha'}] - (2n+2)[\overrightarrow{\beta}]$. Applying Lemma~\ref{lem:RW lemma} to $\overrightarrow{F}$, $\overrightarrow{T_1}$, $\overrightarrow{T_2}$, $c_{1,1}, c_{1,2}, c_{2,1}$ and $c_{2,2}$, there is a horizontal surface $\overrightarrow{g} \colon \overrightarrow{B} \looparrowright \overrightarrow{F} \times S^1$ such that in $ H_1(\overrightarrow{T_1})$ we have
\begin{equation}
\tag{$\dagger$}
\label{eq:ddagger}
\begin{split}
    \bigl[\overrightarrow{g}(\overrightarrow{c_{1}})\bigr] &= [c_{1,1}] =  (2n+1)[\overrightarrow{\alpha}]+(2+ 2n)[\overrightarrow{\beta}]\\
    \bigl[\overrightarrow{g}(\overrightarrow{c_{2}})\bigr] &= [c_{1,2}] = [\overrightarrow{\alpha}]+(2n+2)[\overrightarrow{\beta}]
 \end{split}
\end{equation} and in $H_{1}(\overrightarrow{T_2})$ we have
\begin{align*}
       \bigl[\overrightarrow{g}(\overrightarrow{c_{3}})\bigr] &= [c_{2,1}] = [\overrightarrow{\alpha'}] -(2 +2n)[\overrightarrow{\beta}]\\
    \bigl[\overrightarrow{g}(\overrightarrow{c_{4}})\bigr] &= [c_{2,2}] =  (2n+1)[\overrightarrow{\alpha'}] -(2+2n)[\overrightarrow{\beta}]
\end{align*}

Since $\overleftarrow{g}(\overleftarrow{c_j}) = \overrightarrow{g}(\overrightarrow{c_j})$ in $T$ with $j =1,2$, we may paste horizontal surfaces $\overleftarrow{g} \colon \overleftarrow{B} \looparrowright \overleftarrow{F} \times S^1$ and  $\overrightarrow{g} \colon \overrightarrow{B} \looparrowright \overrightarrow{F} \times S^1$ to form a horizontal surface $g \colon B_n \looparrowright N'$ where $B_n$ is formed from $\overleftarrow{B}$ and $\overrightarrow{B}$ by gluing. The surface $B_n$ has two boundary components $\overrightarrow{c_3}$ and $\overrightarrow{c_4}$. We denote $c_1$ to be the closed curve in $B_n$ obtained from gluing $\overleftarrow{c_1}$ to $\overrightarrow{c_1}$. We denote $c_2$ to be is the closed curve in $B_n$ obtained from gluing $\overleftarrow{c_2}$ to $\overrightarrow{c_2}$. %Let $T$ be the JSJ torus in $N'$ obtained by gluing the boundary torus $\overleftarrow{T}$ in $\overleftarrow{M}$ to the boundary torus $\overrightarrow{T}$ in $\overrightarrow{M}$.

\begin{figure}[h]
\label{figuremain}
\centering
\def\svgwidth{\columnwidth}
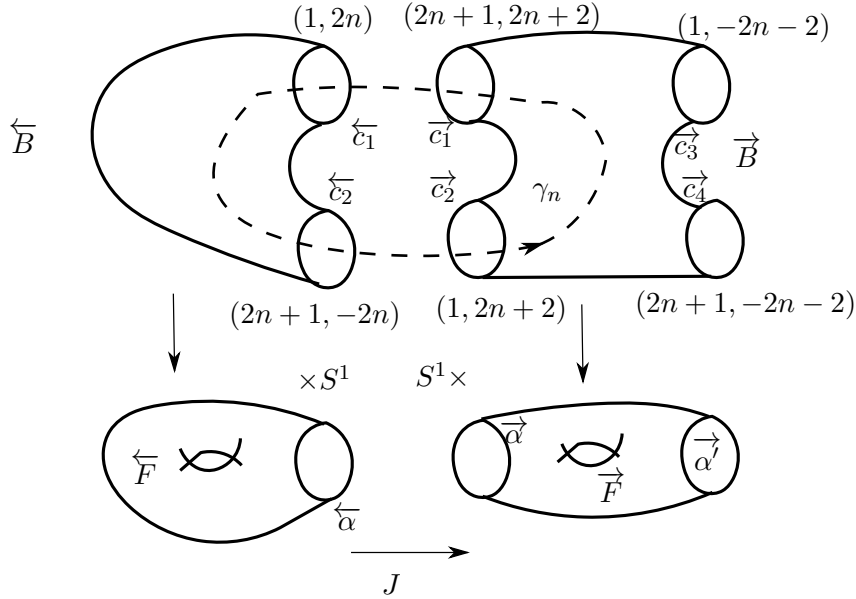
\caption{The left down arrow illustrates the horizontal surface $\protect \overleftarrow{g} \colon \protect \overleftarrow{B} \to \protect \overleftarrow{F} \times {S^{1}}$ and the right down arrow illustrates the horiozntal surface $\protect \overrightarrow{g} \colon \protect \overrightarrow{B} \to  \protect \overrightarrow{F} \times S^1$ in Lemma~\ref{lem:existence}. The simple graph manifold $N'$ is obtained from $\protect \overleftarrow{F} \times S^{1}$ to $\protect \overrightarrow{F} \times S^{1}$ by gluing the boundary torus $\protect \overleftarrow{\alpha} \times S^1$ of $\protect \overleftarrow{F} \times S^1$ to the boundary torus $\protect \overrightarrow{\alpha} \times S^1$ of $\protect \overrightarrow{F} \times S^1$ via the gluing matrix $J$. We paste $\protect \overleftarrow{g}$ and $\protect \overrightarrow{g}$ to form the horizontal surface $g \colon B_{n} \to N'$. The oriented curve $\gamma_n$ in the surface $B_{n}$ is shown in the Figure.}
\end{figure}

Fix a point $x$ in the interior of the subsurface $\overrightarrow{B}$ of $B_n$. Let $\gamma_{n}$ be an oriented simple closed curve in $B_n$ such that starting from $x$ the curve $\gamma_{n}$ intersects each circle $c_1$ and $c_2$ exactly once.
The direction of $\gamma_{n}$ determines directed edges $e_1$ and $e_2$ in the graph $\Omega_{B_n}$.
Applying Lemma~\ref{lem:alternativeslope} to the horizontal surface $g \colon B_n \looparrowright N'$ together with equations (\ref{eq:blacklozenge}) and (\ref{eq:ddagger}), we get that $\mathbf{sl}(e_1) = (2n+1) \bigl/ 1 =2n+1$ and $\mathbf{sl}(e_2) = (2n+1) \bigl/ 1 = 2n+1$.
Therefore, $w(\gamma_n) =  \mathbf{sl}(e_1) \cdot \mathbf{sl}(e_2) = (2n+1)^2 $.

We double the surface $B_n$ along its boundary to get a closed surface, denoted by $S_{n}$. We also double the simple graph manifold $N'$ along its boundary to get a closed simple graph manifold, denoted by $N$. From the horizontal surface $g \colon B_n \looparrowright N'$, after doubling $B_n$ and $N'$ along their respective boundaries, we get a canonical a horizontal surface, denoted by $g_{n} \colon S_{n} \looparrowright N$. Furthermore, since $w(\gamma_n) >1$ it follows that  $g_{n} \colon S_{n} \looparrowright N$ is non-seprable (see Remark~\ref{rem:main}). We note that the manifold $N$ is a closed simple graph manifold with three Seifert blocks and $S_{n}$ is a closed surface with three pieces. We note that  from the construction of $g_{n} \colon S_{n} \looparrowright N$, its governor  is  $ \epsilon_{n} = 2n+1$ and the simple closed curve $\gamma_{n}$ in $S_{n}$ has geometric intersection number with $g_{n}^{-1}(\mathcal{T})$ is $2$ where $\mathcal{T}$ is the union of JSJ tori in $N$.
\end{proof}

To get into the proof of Theorem~\ref{thm:main}, we need several facts from  Section~5 and Section~6 in \cite{Hruska-Nguyen}. Let $\{S_{j} \looparrowright N \}_{j \in \N}$  be the collection of non-separable, horizontal surfaces given by Lemma~\ref{lem:existence}.
We equip $S_{j}$ with a hyperbolic metric, and we equip $N$ with a length metric. These metrics induce metrics on the universal covers $\tilde{S}_j$ and $\tilde{N}$, which are denoted by $d_{\tilde{S}_j}$ and $d$ respectively. In the following, for any two points $x$ and $y$ in $\tilde{N}$ we denote $[x,y]$ as a geodesic in $\tilde{N}$ connecting $x$ to $y$.

Let $S_{n}$ be the surface given by Lemma~\ref{lem:existence}, and let $\gamma_{n}$ be the curve given by Lemma~\ref{lem:existence}. Fact~\ref{fact:1} below is extracted from Section~5 (lower bound of distortion) in \cite{Hruska-Nguyen}, it mainly follows from the proof of Theorem~5.1 in \cite{Hruska-Nguyen}.

\begin{fact}
\label{fact:1} Let $\gamma_n$ be the closed curve in the surface $S_n$ given by Lemma~\ref{lem:existence}. Fix a point $s_0 \in \gamma_{n}$ such that $s_0$ belongs to a circle in the collection of circles in $g_{n}^{-1}(\mathcal{T})$. We relabel $\tilde{s}_{0}$ by $\tilde{x}_0$.
There exists a constants $L \ge 1$ depending on the length of $\gamma_{n}$, and a collection of paths $\bigset{\rho_j}{j \in \N}$ in $\tilde{S}_{n}$ (the path $\rho_j$ is called ``double spiral loop" in \cite{Hruska-Nguyen}) such that the following holds.
\begin{enumerate}
    \item For each $j \in \N$, we have $\tilde{x}_{0} = \rho_{j}(0)$ and $\tilde{x}_{j}:=\rho_{j}(1)$ is an element in the orbit $\pi_1(S_n,s_0)(\tilde{s}_{0})$.
    \item For each $j \in \N$, we have that $[\tilde{x}_0, \tilde{x}_j] -\{\tilde{x}_0, \tilde{x}_j\}$ passes through $4j -1$ Seifert blocks of $\tilde{N}$.
    \item For each $j \in \N$, we have that \[
d(\tilde{x}_0,\tilde{x}_j) \le Lj + L  \,\,\,\text{and} \,\,\, L\,w({\gamma_n})^{j} \le d_{\tilde{S}_{n}}(\tilde{x}_0,\tilde{x}_j).
\]
\end{enumerate}
\end{fact}

The following fact is extracted from Section~6 (upper bound of distortion) in \cite{Hruska-Nguyen}. It mainly follows from Claim~1 and Claim~2 in the proof of Theorem~6.1 in \cite{Hruska-Nguyen}.
\begin{fact}
\label{fact:2}
Let $\epsilon_{m} >1$ be the governor of the non-separable, horizontal surface $g_{m} \colon S_{m} \looparrowright N$. There exists a constant $L' \ge 1$ such that the following holds:
For any $x$ and $y$ in $\tilde{S}_{m}$, let $k$ be the number of Seifert blocks where $[x,y]-\{x,y\}$ passes through. Then 
\[
d_{\tilde{S}_m}(x,y) \le L'\,\epsilon_{m}^{k} + L'd(x,y).
\] 
\end{fact}

Let $N$ and $M$ be closed simple graph manifolds. We equip $N$ and $M$ with length metrics, and these metrics induce the metrics in the universal covers $\tilde{N}$ and $\tilde{M}$, denoted by $d$ and $d'$ respectively. It is shown by Behrstock-Neumann \cite{NeuBeh08} that $\tilde{N}$ and $\tilde{M}$ are quasi-isometric.

\begin{lem}
\label{lem:important}
Let $\varphi \colon \tilde{N} \to \tilde{M}$ be a quasi-isometry. There exists a positive constant $D>0$ such that the following holds. For any two points $x$ and $y$ in $\tilde{N}$ such that $x$ and $y$ belong to JSJ planes of $\tilde{N}$ and $[x,y] - \{x,y\}$ passes through $n$ number of Seifert blocks. Let $k$ be the number of Seifert blocks where $[\varphi(x),\varphi(y)] - \{\varphi(x), \varphi(y)\}$ passes through. Then $k \le n + D$.
\end{lem}

\begin{proof}
By Theorem~1.1 \cite{KapovichLeeb98}, there exists a positive constant $R>$ such that for any Seifert block $B$ in $\tilde{N}$, there exists a Seifert block $B'$ in $\tilde{M}$ such that the Hausdorff distance $d_{\mathcal{H}}(B',\varphi(B)) \le R$. Moreover, for any JSJ plane $P$ in $\tilde{N}$, there exists a JSJ plane $P'$ in  $\tilde{M}$ such that $d_{\mathcal{H}}(\varphi(P),P') \le R$.

Let $\rho >0$ be the infimum of the set of the distance of any two JSJ planes in $\tilde{M}$. We let $D = 2R/\rho  + 2$. We are going to prove that $k \le n +D$.

Let $P_{0}$ and $P_{n}$ be the JSJ planes in $\tilde{N}$ such that $x \in P_0$ and $y \in P_n$. Let $P'_{0}$ and $P'_{n}$ be the JSJ planes in $\tilde{M}$ such that $d_{\mathcal{H}}(\varphi(P_0),P'_0) \le R$ and $d_{\mathcal{H}}(\varphi(P_{n}),P'_{n}) \le R$. It follows that there exist $x' \in P'_{0}$ and $y' \in P'_{n}$ such that $d'(\varphi(x),x') \le R$ and $d'(\varphi(y),y') \le R$.

Let $a$ be the number of Seifert blocks where $[\varphi(x),x'] - \{\varphi(x),x'\}$ passing through. Let $b$ be the number of Seifert blocks where $[\varphi(y),y'] - \{\varphi(y),y'\}$ passing through. We note that the number of Seifert blocks where $[x',y'] - \{x',y'\}$ passing through is no more than $n$. Thus, $k \le a + n +b$.

Since $\rho$ is the smallest distance of any two JSJ planes in $\tilde{M}$, we have $\rho\, (a-1) \le d'(\varphi(x),x')$ and $\rho\, (b-1) \le d'(\varphi(y),y')$. Since $d'(\varphi(x),x') \le R$ and $d'(\varphi(y),y') \le R$, it follows that  $\rho \,a \le R + \rho$ and  $\rho \,b \le R + \rho$ . Hence $a \le R/\rho + 1$ and $b \le R/\rho + 1$. It follows that $k \le n + a +b \le n + 2R/\rho  + 2 = n +D$.
\end{proof}

\section{Proof of the main Theorem}
\label{sec:proofthm}
In this section, we give the proof of Theorem~\ref{thm:main} by showing that there is a infinite collection of natural numbers $\mathcal{F}$ such that the collection of non-separable, horizontal surfaces $\bigset{g_{j} \colon S_{j} \looparrowright N}{j \in \mathcal{F}}$ given by Lemma~\ref{lem:existence} satisfy the conclusion of Theorem~\ref{thm:main}. 

\begin{proof}[Proof of Theorem~\ref{thm:main}]
For each $j \in \N$, let $S_{j} \looparrowright N$ be the non-separable, horizontal surface given by Lemma~\ref{lem:existence}. Let $\gamma_j$ be the simple closed curve given by Lemma~\ref{lem:existence}. Let $\epsilon_j$ be the governor of the horizontal surface $S_j \looparrowright N$. We note that $\epsilon_j = 2j +1$ and $w(\gamma_j) = (2j+1)^2$ by Lemma~\ref{lem:existence}.

Let $\mathcal{F}$ be a infinite collection of natural numbers such that for any elements $n$ and $m$ in $\mathcal{F}$ we have $(2m+1)^{2} < 2n +1$ whenever $m < n$ (the existence of this collection is easy to see, for instance, we may define $\tau(j)$ inductively by letting $\tau(j+1) = \bigl (2\tau(j) +1  \bigr )^2 +1$ and then we may let $\mathcal{F} = \bigset{\tau(j)}{j \in \N}$).

To prove the theorem we only need to show that if $n$ and $m$ are two elements $\mathcal{F}$ such that $m < n$ then $\bigl (\pi_1(N),\pi_1(S_{n}) \bigr )$ and $\bigl (\pi_1(N),\pi_1(S_m) \bigr )$ are not quasi-isometric. We prove this by contradiction. We briefly describe here how do we get a contradiction. Suppose that $\bigl (\pi_1(N),\pi_1(S_{n}) \bigr )$ and $\bigl (\pi_1(N),\pi_1(S_m) \bigr )$ are quasi-isometric, then we are going to show $w(\gamma_n) \le \epsilon_{m}^{4}$. From this inequality and the facts $\epsilon_j = 2j +1$ and $w(\gamma_j) = (2j+1)^2$, we get that $2n +1 \le (2m+1)^2$. Since $m, n \in \mathcal{F}$ and $m <n$, it follows that $(2m+1)^2 < 2n+ 1$. The contradiction comes from two inequalities $2n +1 \le (2m+1)^2$ and $(2m+1)^2 < 2n+ 1$.

We fix a finite generating set $\mathcal{A}_n$ of $\pi_1(S_n)$, a finite generating set $\mathcal{A}_{m}$ of $\pi_1(S_m)$ and a finite generating set $\mathcal{S}$ of $\pi_1(N)$ so that $\mathcal{A}_n \subset \mathcal{S}$ and $\mathcal{A}_m \subset \mathcal{S}$. We note that $\mathcal{S}$ depends on the choice of $m$ and $n$.
Assume that $\bigl (\pi_1(N),\pi_1(S_{n}) \bigr )$ and $\bigl (\pi_1(N),\pi_1(S_m) \bigr )$ are quasi-isometric, it follows that $\bigl (\tilde{N}, \tilde{S}_m \bigr )$ and $(\bigl (\tilde{N},\tilde{S}_n \bigr )$ are quasi-isometric by Lemma~\ref{cor: geometric relative quasi to group relative quasi}. Hence, there exists a  positive constant $L_1$ and an $(L_1,L_1)$--quasi-isometry map $\varphi \colon \tilde{N} \to \tilde{N}$ such that $\varphi(\tilde{S}_n) \subseteq \tilde{S}_m$ and $\tilde{S}_m \subseteq \mathcal{N}_{L_1}\bigl (\varphi(\tilde{S}_n) \bigr )$.

Let $L \ge 1$ be the constant given by Fact~\ref{fact:1} with respect to the horizontal surface $S_{n} \looparrowright N$. Let $L' \ge 1$ be the constant given by Fact~\ref{fact:2} with respect to the horizontal surface $S_{m} \looparrowright N$.
Let $A = \max \{L,L',L_1\}$.

Let $\bigset{\tilde{x}_{j}}{j \in \N}$ be the collection of points given by Fact~\ref{fact:1}. Let $D$ be the constant given by Lemma~\ref{lem:important}. We first claim that
\begin{equation}
    \tag{$**$}
    \label{eq:doublestar}
  d_{\tilde{S}_{m}}\bigl (\varphi(\tilde{x}_0),\varphi(\tilde{x}_j) \bigr ) \le A \epsilon_{m}^{4j +D} + A^{3}j + A^{3} + A^{2}  
\end{equation}
Indeed, From Fact~\ref{fact:1} we have
\[
d(\tilde{x}_0,\tilde{x}_j) \le Aj + A  \,\,\,\text{and} \,\,\, w({\gamma_n})^{j} \le A\,d_{\tilde{S}_{n}}(\tilde{x}_0,\tilde{x}_j)
\]

Using the above inequality and the fact $\varphi$ is a $(A,A)$--quasi-isometry, we get that
\[
d\bigl (\varphi(\tilde{x}_0),\varphi(\tilde{x}_j) \bigr ) \le A\,d(\tilde{x}_0,\tilde{x}_j) + A \le A(Aj + A) + A = A^{2}j + A^{2} + A
\]

We recall that $[\tilde{x}_0, \tilde{x}_j] - \{\tilde{x}_0, \tilde{x}_j\}$ passes through $2j$ Seifert blocks of $\tilde{N}$. 
Let $[\varphi(\tilde{x}_0),\varphi(\tilde{x}_j)]$ be a geodesic in $\tilde{N}$ connecting $\varphi(\tilde{x}_0)$ to $\varphi(\tilde{x}_j)$. Let $k$ be the number of Seifert blocks of $\tilde{N}$ where $[\varphi(\tilde{x}_0),\varphi(\tilde{x}_j)] - \{\varphi(\tilde{x}_0),\varphi(\tilde{x}_j)\}$ passes through. By Lemma~\ref{lem:important}, we have that $k \le 2j +D$.
Using Fact~\ref{fact:2}, the above inequality $d\bigl (\varphi(\tilde{x}_0),\varphi(\tilde{x}_j) \bigr ) \le A^{2}j + A^{2} + A$, and $k \le 4j +D$ we get that
\begin{align*}
  d_{\tilde{S}_{m}}\bigl (\varphi(\tilde{x}_0),\varphi(\tilde{x}_j) \bigr ) &\le A\,\epsilon_{m}^{k} + Ad\bigl (\varphi(\tilde{x}_0),\varphi(\tilde{x}_j) \bigr ) \\
  &\le A\,\epsilon_{m}^{k} + A(A^{2}j + A^{2} + A)\\
  &\le A\,\epsilon_{m}^{k} + A^{3}j + A^{3} + A^{2} \le A\epsilon_{m}^{4j+D} +A^{3}j +A^{3} + A^{2}\
\end{align*}
Thus (\ref{eq:doublestar}) is established.

By Lemma~\ref{lem:invariant of length}, there exists constant $\xi >0$ such that
\[
d_{\tilde{S}_n}(\tilde{x}_0, \tilde{x}_j) \le \xi \,d_{\tilde{S}_m}\bigl (\varphi(\tilde{x}_0),\varphi(\tilde{x}_j) \bigr ) + \xi
\] for all $j$.

We use Fact~\ref{fact:1}, the above inequality, and (\ref{eq:doublestar})
to get that \begin{align*}
    w({\gamma_n})^{j} &\le A\,d_{\tilde{S}_n}(\tilde{x}_0, \tilde{x}_j)\\ 
    &\le A\, \bigl ( \xi d_{\tilde{S}_m}\bigl (\varphi(\tilde{x}_0),\varphi(\tilde{x}_j) \bigr ) + \xi \bigr )\\
    &= A\xi d_{\tilde{S}_m}\bigl (\varphi(\tilde{x}_0),\varphi(\tilde{x}_j) \bigr ) + A\xi\\
    &\le A\xi \bigl ( A\epsilon_{m}^{k} + A^{3}j + A^{3} + A^{2} \bigr ) + A\xi \\  
    &\le A\xi \bigl ( A\epsilon_{m}^{4j+D} + A^{3}j + A^{3} + A^{2} \bigr ) + A\xi \\
    &= A^{2}\xi\epsilon_{m}^{4j+D} + A^{4}\xi j + A^{4}\xi + A^{3}\xi + A\xi\
\end{align*}
We divide both sides of the inequality \[w({\gamma_n})^{j} \le A^{2}\xi\epsilon_{m}^{4j+D} + A^{4}\xi j + A^{4}\xi + A^{3}\xi + A\xi \] by $\epsilon_{m}^{4j}$ to get that
\[
\Bigl(w(\gamma_n) \bigl/\epsilon_{m}^4\Bigr)^{j} \le A^{2}\xi\epsilon_{m}^{D} + A^{4}\xi j \bigl / \epsilon_{m}^{4j} + \bigl ( A^{4}\xi + A^{3}\xi + A\xi \bigr ) \bigl / \epsilon_{m}^{4j}
\] for all $j \in \N$.

Since $\epsilon_{m} >1$, we have 
\[\lim_{j \to \infty} \Bigl(A^{2}\xi\epsilon_{m}^{D} + A^{4}\xi j \bigl / \epsilon_{m}^{4j} + \bigl ( A^{4}\xi + A^{3}\xi + A\xi \bigr ) \bigl / \epsilon_{m}^{4j} \Bigr) = A^{2}\xi\epsilon_{m}^{D}.
\]
Hence $\lim_{j \to \infty}\Bigl(w(\gamma_n) \bigl/\epsilon_{m}^4\Bigr)^{j} \le A^{2}\xi\epsilon_{m}^{D}$. It follows that $w(\gamma_n) \bigl / \epsilon_{m}^{4} \le 1$, otherwise we will get $\infty \le A^{2}\xi\epsilon_{m}^{D}$. Thus 
\[
w(\gamma_n) \le \epsilon_{m}^{4}
\]
Since $w(\gamma_n) = (2n+1)^2$ and $\epsilon_m = 2m+1$, it follows that $(2n+1)^2 \le (2m+1)^4$. Hence $2n +1 \le (2m+1)^2$. 

We note that $m <n$ and $m,n \in \mathcal{F}$, thus by the definition of $\mathcal{F}$ we have $(2m+1)^2 < 2n+1$. Combining two inequalities $2n +1 \le (2m+1)^2$ and $2n +1 \le (2m+1)^2$, we have $n <n$, a contradiction.
 The theorem is established.
\end{proof}

\section{Appendix}
In this section, we give an evidence supporting Conjecture~\ref{conj} by showing that other geometric invariants in literature such as subgroup distortion, $k$--volume distortion, relative upper divergence, relative lower divergence could not be used to distinguish quasi-isometry of pairs of separable, horizontal surfaces in graph manifolds.

\subsection{$k$--volume distortion}
$k$--volume distortion ($k \ge 1$) is a notion introduced by Bennett \cite{Bennett2011}. We remark that this notion agrees with subgroup distortion when $k=1$ and area distortion (introduced by Gersten \cite{GerstenArea}) when $k=2$. We refer the reader to \cite{Bennett2011} for a precise definition of $k$--volume distortion. 
\begin{prop}
Let $S \looparrowright N$ be a separable, horizontal surface in a graph manifold $N$. Then the $k$--volume distortion of $\pi_1(S)$ in $\pi_1(N)$ is quadratic when $k=1$, is linear when $k=2$ and is trivial when $k \ge 3$.
\end{prop}
\begin{proof}
$1$--volume distortion (i.e, subgroup distortion) of $\pi_1(S)$ in $\pi_1(N)$ is quadratic (see \cite{Hruska-Nguyen}).  
We are going to show that $2$--volume distortion (i.e, area subgroup distortion) of $\pi_1(S)$ in $\pi_1(N)$ is linear. Indeed, the paragraph after Proposition~5.4 in \cite{GerstenArea} shows that if the Dehn function of $\pi_1(S)$ is linear then $2$--volume distortion of $\pi_1(S)$ in $\pi_1(N)$ is linear. Since Dehn function of the fundamental group of a hyperbolic surface is linear. The claim is confirmed. Finally, we consider the case $k \ge 3$. Since there is no $k$--cell in the universal cover $\tilde{S}$, it follows from the definition of $k$--volume distortion that $k$--volume distortion ($k \ge 3$) of $\pi_1(S)$ in $\pi_1(N)$ is trivial.
\end{proof}

\subsection{Relative divergence}
In \cite{Hung15}, Tran introduces the notions of relative upper divergence and relative lower divergence of a pair of finitely generated groups $H\le G$, denoted by $Div(G,H)$ and $div(G,H)$ respectively, and shows that relative upper divergence and relative lower divergence are quasi-isometric invariants (see Proposition~4.3 and Proposition~4.9 in \cite{Hung15}). Since relative upper divergence and relative lower divergence are quite technical and we only use results established in \cite{Hung15}, \cite{Tran17}, we refer the reader to \cite{Hung15} for a precise definition.
\begin{prop}
\label{prop:div}
Let $g\colon{(S,s_0)}\looparrowright{(N,x_0)}$ be a separable, horizontal surface in a graph manifold $N$. Let $G = \pi_1(N,x_0)$ and $H = \pi_1(S,s_0)$. Then $Div(G,H)$ is quadratic and $div(G,H)$ is linear.
\end{prop}

If a horizontal surface $g \colon S \looparrowright N$ is separable then there exist finite covers $\hat{S}\to S$ and $\hat{N}\to N$ such that $\hat{N}$ is an $\hat{S}$--bundle over $S^1$ (see  \cite{WangYu}).  Relative upper divergence  and relative lower divergence are unchanged when passing to subgroups of finite index, so for the rest of this section, without of loss generality we assume the graph manifold $N$ fibers over $S^1$ with the fiber $S$. We remark that $N$ is the mapping torus of a homeomorphism $f \in \Aut(S)$.
In particular, if we let $G = \pi_1(N)$ and $H = \pi_1(S)$, then $G = H \rtimes_\phi \Z$, where $\phi \in \Aut(H)$ is an automorphism induced by $f$. We note that that the distortion of $\pi_1(S)$ in $\pi_1(N)$ is quadratic as $S$ is embedded in $N$.

The proof of Proposition~\ref{prop:div} is a combination of Lemma~\ref{lem:uperbound of Div} and Lemma~\ref{lem:lowerdiv}.
\begin{lem}
\label{lem:uperbound of Div}
 $Div(G,H)$ is at most quadratic  and $div(G,H)$ is linear.
\end{lem}
\begin{proof}
We note that $G = H \rtimes_\phi \Z$, where $\phi \in \Aut(H)$. We fix finite generating sets $\mathcal{A}$ and $\mathcal{B}$ of $H$ and $G$ respectively.
By Proposition~4.3 in \cite{Tran17} we have $Div(G,H) \preceq \Delta_{H}^{G}$. Since $H$ is quadratically distorted in $G$, it follows that $Div(G,H) \preceq n^2$.

To see that $div(G,H)$ is linear, it suffices to show $div(G,H)$ is dominated by a linear function (because $div(G,H)$ is always bounded below by a linear function).
Since $H$ is a normal subgroup in $G$, by Theorem~5.4 in \cite{Hung15} we have $div(G,H) \preceq dist^{G}_{H}$ where $dist^{G}_{H}(n) =\min \bigset{|h|_{\mathcal{A}}}{h \in H, |h|_{\mathcal{B}} \ge n}$. We fix a circle $c$ in $g^{-1}(\mathcal{T})$ where $\mathcal{T}$ is the collection of JSJ tori of $N$. Let $K = \pi_1(c)$. We note that $K$ is undistorted in $G$. By Theorem~3.6 and Proposition~3.5 in \cite{Hung15}, we have $dist^{G}_{H} \preceq dist^{G}_{K} \preceq \Delta^{G}_{K}$. It follows that $dist^{G}_{H}$ is dominated by a linear function because $K$ is undistorted in $G$.
\end{proof}

\begin{defn}
The  divergence of a bi-infinite quasi-geodesic $\alpha$, denoted by $Div(\alpha)$, is a function $g \colon (0,\infty) \to (0,\infty)$ which for each positive number $r$ the value $g(r)$ is the infimum on the lengths of all paths outside the open ball with radius $r$ about $\alpha(0)$ connecting $\alpha(-r)$ to $\alpha(r)$.
\end{defn}

The following lemma  is proved implicitly in \cite{Tran17}. We use it in the proof of Lemma~\ref{lem:lowerdiv}.
\begin{lem}
\label{lem: lowerbound Div}
 Suppose that there exists an element $h$ in $H$ with infinite order such that the map $\alpha \colon \Z \to G$ determined by $\alpha(n) = h^n$ is an $(L,C)$--quasi-isometric embedding. Then $Div(\alpha) \preceq Div(G,H)$.
\end{lem}

\begin{lem}
\label{lem:lowerdiv}
$Div(G,H)$ is at least quadratic.
\end{lem}

\begin{proof}
We equip $N$ with a Riemannian metric and this metric induces a metric on the universal cover $\tilde{N}$, denoted by $d$. We equip $S$ with a hyperbolic metric and this metric induces a metric on the universal cover $\tilde{S}$, denoted by $d_{\tilde{S}}$.

Let $\mathcal{T}$ be the JSJ decoposition of $N$.
Choose a geodesic loop $\gamma$ such that $\gamma$ and $g^{-1}(\mathcal{T})$ has non-trivial geometric intersection number (see Lemma~3.3 in \cite{Hruska-Nguyen} for the existence of such a loop $\gamma$). We also assume that $s_0 \in \gamma$. Let $h = [\gamma] \in \pi_1(S,s_0)$. We note that $h$ has infinite order. Let $\alpha \colon \Z \to G$ be determined by $\alpha(n) = h^{n}$, and let $\beta \colon \Z \to \tilde{N}$ be determined by $\beta(n) = h^{n} \cdot \tilde{s}_0$  for each $n \in \Z$.  We will show that $\beta$ is a quasi-geodesic and $Div(\beta)$ is at least quadratic and thus it follows that $\alpha$ is a quasi-geodesic and  $Div(\alpha)$ is at least quadratic. We then apply Lemma~\ref{lem: lowerbound Div} to get that $Div(G,H)$ is at least quadratic.

We are going to show $\beta$ is a quasi-geodesic. Let $\tilde{\gamma}$ be the path lift of $\gamma$ based at $\tilde{s}_{0}$. Let $k$ be the number of Seifert blocks of $\tilde{N}$ where $\tilde{\gamma}$ passes through. It follows that a geodesic $[h^{n}\cdot \tilde{s}_{0},h^{m}\cdot \tilde{s}_{0}]$ in $\tilde{N}$ passes through $k\abs{m-n}$ Seifert blocks in $\tilde{N}$. Let $\rho$ be the shortest distance of any two JSJ planes in $\tilde{N}$. It follows that $\rho k\abs{m-n} \le d(h^{n}\cdot \tilde{s}_{0},h^{m}\cdot \tilde{s}_{0}) = d(\beta(n),\beta(m))$. Since $\pi_1(S)$ is a hyperbolic group, it follows that there is $A>0$ such that $\beta$ is an $(A,A)$--quasi-geodesic  in $\tilde{S}$ with respect to $d_{\tilde{S}}$--metric. Hence, for any $n,m \in \Z$ we have $d_{\tilde{S}}(\beta(n),\beta(m)) \le A\abs{m-n} + A$. Since $d(\beta(n),\beta(m)) \le d_{\tilde{S}}(\beta(n),\beta(m))$, it follows that $d(\beta(n),\beta(m)) \le A\abs{m-n} + A$. Let $L = \max \{A, 1/k\rho\}$, we easily see that $\beta$ is an $(L,L)$--quasi-geodesic.

We are now going to show $Div(\beta)$ is at least quadratic. Lift the JSJ decomposition of the graph manifold $N$ to the universal cover $\tilde{N}$, and let $T_N$ be the tree dual to this decomposition of $\tilde{N}$.
We note that $h$ acts hyperbolically on the tree $T_{N}$ in the sense that there exists a vertex $v \in T_N$ and there exists a bi-infinite geodesic $\gamma$ in $T_N$ such that $\bigset{h^{j}v}{j \in \Z}$ is an unbounded subset of $\gamma$.
 By Proposition~3.7 in \cite{Sisto11}, it follows that $h$ is a contracting element in $\pi_1(N)$, and hence $h$ is Morse element in $\pi_1(N)$ (see Lemma~2.9 in \cite{Sisto11}). Thus, $\beta$ is a Morse quasi-geodesic in $(\tilde{N},d)$. By Theorem~1.1 in \cite{KL98}, there exists a $\CAT(0)$ space $(X,d')$ such that $(\tilde{N},d)$ and $(X,d')$ are bilipschitz homeomorphism. It shown in \cite{BD14} (see also in \cite{Sultan14}) that divergence of a Morse bi-infinite quasi-geodesic is at least quadratic. Hence, the divergence of the image of $\beta$ in $(X,d')$ under the bilipschitz homeomorphism $(\tilde{N},d) \to (X,d')$ is at least quadratic. It follows that $Div(\beta)$ in $\tilde{N}$ is at least quadratic.
\end{proof}

\begin{proof}[Proof of Proposition~\ref{prop:div}] 
The proof is a combination of Lemma~\ref{lem:uperbound of Div} and Lemma~\ref{lem:lowerdiv}.
\end{proof}

\bibliographystyle{alpha}
\bibliography{hoang}

\begin{thebibliography}{MSW11}

\bibitem[BDt14]{BD14}
Jason Behrstock and Cornelia Dru\c~tu.
\newblock Divergence, thick groups, and short conjugators.
\newblock {\em Illinois J. Math.}, 58(4):939--980, 2014.

\bibitem[Ben11]{Bennett2011}
Hanna Bennett.
\newblock Volume distortion in groups.
\newblock {\em Algebr. Geom. Topol.}, 11(2):655--690, 2011.

\bibitem[BN08]{NeuBeh08}
Jason~A. Behrstock and Walter~D. Neumann.
\newblock Quasi-isometric classification of graph manifold groups.
\newblock {\em Duke Math. J.}, 141(2):217--240, 2008.

\bibitem[DtK18]{DK18}
Cornelia Dru\c~tu and Michael Kapovich.
\newblock {\em Geometric group theory}, volume~63 of {\em American Mathematical
  Society Colloquium Publications}.
\newblock American Mathematical Society, Providence, RI, 2018.
\newblock With an appendix by Bogdan Nica.

\bibitem[Ger96]{GerstenArea}
S.~M. Gersten.
\newblock Preservation and distortion of area in finitely presented groups.
\newblock {\em Geom. Funct. Anal.}, 6(2):301--345, 1996.

\bibitem[HN17]{Hruska-Nguyen}
G.~C. {Hruska} and T.~H. {Nguyen}.
\newblock {Distortion of surfaces in graph manifolds}.
\newblock Preprint. arXiv:1703.07458 [math.GR], March 2017.

\bibitem[KL97]{KapovichLeeb98}
Michael Kapovich and Bernhard Leeb.
\newblock Quasi-isometries preserve the geometric decomposition of {H}aken
  manifolds.
\newblock {\em Invent. Math.}, 128(2):393--416, 1997.

\bibitem[KL98]{KL98}
M.~Kapovich and B.~Leeb.
\newblock {$3$}-manifold groups and nonpositive curvature.
\newblock {\em Geom. Funct. Anal.}, 8(5):841--852, 1998.

\bibitem[Liu17]{YiLiu2017}
Yi~Liu.
\newblock A characterization of virtually embedded subsurfaces in 3-manifolds.
\newblock {\em Trans. Amer. Math. Soc.}, 369(2):1237--1264, 2017.

\bibitem[MSW11]{MSW11}
Lee Mosher, Michah Sageev, and Kevin Whyte.
\newblock Quasi-actions on trees {II}: {F}inite depth {B}ass-{S}erre trees.
\newblock {\em Mem. Amer. Math. Soc.}, 214(1008):vi+105, 2011.

\bibitem[RW98]{RW98}
J.~Hyam Rubinstein and Shicheng Wang.
\newblock {$\pi_1$}-injective surfaces in graph manifolds.
\newblock {\em Comment. Math. Helv.}, 73(4):499--515, 1998.

\bibitem[Sch95]{Schwartz95}
Richard~Evan Schwartz.
\newblock The quasi-isometry classification of rank one lattices.
\newblock {\em Inst. Hautes \'Etudes Sci. Publ. Math.}, (82):133--168 (1996),
  1995.

\bibitem[{Sis}11]{Sisto11}
A.~{Sisto}.
\newblock {Contracting elements and random walks}.
\newblock Preprint. arXiv:1112.2666 [math.GT], December 2011.

\bibitem[Sul14]{Sultan14}
Harold Sultan.
\newblock Hyperbolic quasi-geodesics in {CAT}(0) spaces.
\newblock {\em Geom. Dedicata}, 169:209--224, 2014.

\bibitem[Tra15]{Hung15}
Hung~Cong Tran.
\newblock Relative divergence of finitely generated groups.
\newblock {\em Algebr. Geom. Topol.}, 15(3):1717--1769, 2015.

\bibitem[Tra17]{Tran17}
Hung~Cong Tran.
\newblock Geometric embedding properties of {B}estvina-{B}rady subgroups.
\newblock {\em Algebr. Geom. Topol.}, 17(4):2499--2510, 2017.

\bibitem[Woo16]{Woodhouse16}
D.J. Woodhouse.
\newblock Classifying finite dimensional cubulations of tubular groups.
\newblock {\em Michigan Math. J.}, 65(3):511--532, 2016.

\bibitem[WY97]{WangYu}
S.~Wang and F.~Yu.
\newblock Graph manifolds with non-empty boundary are covered by surface
  bundles.
\newblock {\em Math. Proc. Cambridge Philos. Soc.}, 122(3):447--455, 1997.

\end{thebibliography}
\end{document}